\documentclass[12pt]{amsart}
\usepackage{a4wide}
\usepackage{amsmath,amssymb,amsfonts}
\usepackage{float,epsfig,hyperref,graphicx,graphics,mathrsfs,dsfont}
\usepackage{color,cite}

\newtheorem{example}{Example}[section]

\newtheorem{thm}{Theorem}[section]

\newcommand {\mat}  [1] {\left[\begin{array}{#1}}
\newcommand {\rix}      {\end{array}\right]}
\begin{document}
\pagestyle{plain}

\title{Graph Directed Coalescence Hidden Variable Fractal Interpolation Functions}
\author{Md.~Nasim Akhtar and M.~Guru Prem Prasad \\
Department of Mathematics \\
Indian Institute of Technology Guwahati\\
Guwahati~781039, Assam, India \\
}

\thanks{Correspondence Email: nasim@iitg.ernet.in, nasim.iitm@gmail.com}
\date{}
\maketitle
\begin{abstract}
Fractal interpolation function (FIF) is a special type of continuous function which interpolates certain data set and  the attractor of the Iterated function system (IFS) corresponding to the data set is the graph of the FIF. Coalescence Hidden-variable Fractal Interpolation Function (CHFIF) is both self-affine and non self-affine in nature depending on the free variables and constrained free variables for a generalized IFS. In this article graph directed iterated function system for a finite number of generalized data sets is considered and  it is shown that the projections of the attractors on $\mathbb{R}^{2}$ is the graph of the CHFIFs interpolating the corresponding data sets.

\end{abstract}
\section{Introduction}
The concept of fractal interpolation function (FIF) based on an iterated function system (IFS)  as a fixed point of  Hutchinson's operator is introduced by Barnsley \cite{barnsley-1986, barnsley-1988}. The attractor of the IFS is the graph of the fractal function interpolating certain data set. These FIFs are generally self-affine in nature.
The idea has been extended
to a generalized data set in $\mathbb{R}^{3}$ such that the projection of the graph of the corresponding FIF
onto $\mathbb{R}^{2}$  provides a non self-affine interpolation function namely Hidden variable FIFs for a given data set $\{(x_{n}, y_{n}): n=0, 1, \ldots, N\}$ \cite{barnsley-1989a}.
Chand and Kapoor \cite{chand-2007}, introduced the concept of Coalescence hidden variable FIFs which are both self-affine and non self-affine for generalized IFS. The extra degree of freedom is useful to adjust the shape and fractal dimension of the interpolation functions.
In \cite{barnsley-1989},  Barnsley  et al. proved  existence of a differentiable FIF.
The continuous but nowhere differentiable fractal function namely $\alpha$-fractal interpolation function $f^{\alpha}$ is introduced by Navascues as perturbation of a continuous function $f$ on a compact interval $I$ of $\mathbb{R}$~\cite{navascues-2005b, navascues-2005}. Interested reader can see for the theory and application of  $\alpha$-fractal interpolation function $f^{\alpha}$
which has been extensively explored by Navascues \cite{navascues-2010a,  navascues-2011, navascues-2005}.
\par
In \cite{deniz-15} Deniz et al. considered graph-directed iterated function system for finite number of data sets and proved the existence of fractal functions  interpolating corresponding data sets with graphs as the attractor of the GDIFS.
\par
In the present work, generalized GDIFS for generalized interpolation data sets in $\mathbb{R}^{3}$ has taken. It is shown that, corresponding to the data sets there exists CHFIFs whose graph is the projection on $\mathbb{R}^{2}$ of the attractors of the GDIFS.
\section{Preliminaries}
\subsection{Iterated Function System}
Let $\mathcal{X}\subset\mathbb{R}^{n}$ and $(\mathcal{X}, d_{\mathcal{X}})$ be a complete metric space. Also assume $\mathcal{H}(\mathcal{X})= \{S\subset \mathcal{X}; S\neq \Phi, S \; \mbox{is compact in} \; \mathcal{X}\}$ with the Hausdorff metric $d_{\mathcal{H}}(A, B)$ defined as $d_{\mathcal{H}}(A, B)= \max\{d_{\mathcal{X}}(A, B), d_{\mathcal{X}}(B, A)\}$, where $d_{\mathcal{X}}(A, B)=\max_{x\in A}\min_{y\in B}d_{\mathcal{X}}(x, y)$ for any two sets $A, B$ in $\mathcal{H}(\mathcal{X})$. $(\mathcal{H}, d_{\mathcal{H}})$ is a complete metric space whenever $(\mathcal{X}, d_{\mathcal{X}})$ is complete. Let for $i=1, \; 2, \; \ldots, \; N$, $w_{i}: \mathcal{X}\to \mathcal{X}$ are continuous maps then $\{\mathcal{X}; w_{i}: i=1, \; 2, \; \ldots, \; N\}$ is called an iterated function system (IFS). If the maps $w_{i}$'s are contraction then, the set valued Hutchinson operator $W:\mathcal{H}(\mathcal{X})\to \mathcal{H}(\mathcal{X})$ defined by $W(B)=\bigcup_{i=1}^{N}w_{i}(B)$, where $w_{i}(B):= \{w_{i}(b): b\in B\}$ is also contraction. Then by Banach fixed point theorem, there exists a unique set $A\in\mathcal{H}(\mathcal{X})$ such that $A=W(A)=\bigcup_{i=1}^{N}w_{i}(A)$. The set $A$ is called the attractor associated with the IFS $\{\mathcal{X}; w_{i}: i=1, \; 2, \; \ldots, \; N\}$.
\subsection{Fractal Interpolation Function}
Let a set of interpolation points $\{(x_{i},y_{i}) \in I\times\mathbb{R} \; : \; i=0, \; 1, \; \dotsc, \;  N\}$ be given,
where  $\Delta: \, x_{0}<x_{1}< \dotsb <x_{N}$ is a partition of the closed interval
$I=[x_{0}, \, x_{N}]$ and $y_{i}\in [g_{1},g_{2}]\subset \mathbb{R}$, $i=0$, $1$, $\dotsc$, $N$.
Set $I_{i}=[x_{i-1},x_{i}]$ for $i=1$, $2$, $\dotsc$, $N$ and $K=I\times [g_{1},g_{2}]$.
Let $L_{i}: I\to I_{i}, i=1, \; 2, \; \dotsc,  N$, be contraction homeomorphisms such that
  \begin{equation}\label{l1}
   L_{i}(x_{0})=x_{i-1}, \; \; \; L_{i}(x_{N})=x_{i},
   \end{equation}
   \begin{equation}
   |L_{i}(c_{1})-L_{i}(c_{2})|\leq d|c_{1}-c_{2}| \; \mbox{for all}  \; c_{1} \; \mbox{and} \;  c_{2} \; \mbox{in} \; I , \label{l2}
   \end{equation}
   for some $0\leq d<1$. Furthermore, let  $H_{i}: K\rightarrow \mathbb{R}$,  $i=1, \; 2, \; \dotsc,  N$,  be given continuous functions such that
   \begin{equation}
   H_{i}(x_{0},y_{0})=y_{i-1}, \; \; \; H_{i}(x_{N},y_{N})=y_{i}, \label{f01}
   \end{equation}
   \begin{equation}
   |H_{i}(x,\xi_{1})-H_{i}(x,\xi_{2})|\leq |\alpha_{i}||\xi_{1}-\xi_{2}|  \label{f2}
   \end{equation}
   for all $x\in I$ and for all $\xi_{1}$ and $\xi_{2}$ in  $[g_{1},g_{2}]$,
   for some $\alpha_{i}\in (-1,1), i=1, \; 2, \; \ldots,  N$.
   Define mappings $W_{i}: K\to I_{i}\times \mathbb{R}, i=1, \; 2, \; \dotsc,  N$ by
   $$W_{i}(x,y)=( L_{i}(x), H_{i}(x,y)) \; \mbox{for all} \; (x,y)\in K\; .$$
   Then
   \begin{equation}
   \{K;\; W_{i}(x,y)\; :  \; i=1, \; 2, \; \dotsc,  N\}  \; \label{ifs1}
   \end{equation}
   constitutes an IFS.
 Barnsley \cite{barnsley-1986} proved that the IFS $\{K; W_{i}: i=1, \; 2, \; \dotsc, \;  N \}$
 defined above has a unique attractor $G$ where $G$ is the graph of a continuous function
 $f:I\to \mathbb{R}$ which obeys  $f(x_{i})=y_{i}$ for $i=0$, $1$, $\ldots$, $N$.
 This function $f$ is called a fractal interpolation function (FIF) or simply fractal function
 and it is the unique function satisfying the following fixed point equation
   \begin{equation}
   f(x)= F_{i}(L_{i}^{-1}(x),f(L_{i}^{-1}(x)))\; \mbox{for all} \;  x\in I_{i}, \; i=1, \; 2, \; \ldots,  N. \label{fix}
   \end{equation}
   The widely studied FIFs so far are defined by the iterated mappings
   \begin{equation}
    L_{i}(x)= a_{i}\; x\; + \; d_{i}, \; \;  F_{i}(x,y)= \alpha_{i} \; y \; + q_{i}(x), \; \; i=1, \; 2, \; \dotsc,  N, \label{ifsf}
    \end{equation}
   where the real constants $a_{i}$ and $d_{i}$ are determined by the condition \eqref{l1} as
  \begin{equation}
   a_{i}=\frac{(x_{i}-x_{i-1})}{(x_{N}-x_{0})} \; \; \mbox{and}\; \; d_{i}= \frac{(x_{N}x_{i-1}-x_{0}x_{i})}{(x_{N}-x_{0})} , \label{a1}
   \end{equation}
   and $q_{i}(x)$'s are suitable continuous functions such that the conditions \eqref{f01} and \eqref{f2} hold.
   For each $i$,
    $\alpha_{i}$ is a free parameter with $\vert \alpha_{i} \vert < 1$ and is called a vertical scaling
    factor of the transformation $W_{i}$.  Then the vector $\alpha=(\alpha_{1},\alpha_{2}, \; \ldots, \;  \alpha_{N})$
    is called the scale vector of the IFS. If   $q_{i}(x)$ is taken as linear then the corresponding FIF is known as affine FIF (AFIF).
\subsection{Coalescence FIF}
To construct a Coalescence Hidden-variable Fractal Interpolation Functions, a set of real parameters $z_{i}$ for $i=1, \; 2, \; \dotsc,  N$ are introduced and the generalized interpolation data $\{(x_{i},y_{i},z_{i}) \in\mathbb{R}^{3} \; : \; i=0, \; 1, \; \dotsc, \;  N\}$ is considered. Then define the maps $w_{i}: I\times \mathbb{R}^{2}\to I_{i}\times \mathbb{R}^{2}, i=1, \; 2, \; \dotsc,  N$ by
   $$w_{i}(x,y,z)=( L_{i}(x), F_{i}(x,y,z)) $$
   where, $L_{i}: I\to I_{i}, i=1, \; 2, \; \dotsc,  N$ are given in (\ref{ifsf}), and the functions $ F_{i}:I\times \mathbb{R}^{2}\to\mathbb{R}^{2}$ such that  $F_{i}(x,y,z)=(F_{i}^{1}(x,y,z), F_{i}^{2}(x,y,z))=(\alpha_{i}y+\beta_{i}z+ c_{i}x+ d_{i}, \gamma_{i}z+ e_{i}x+ f_{i})$ satisfy the join-up conditions
$$F_{i}(x_{0},y_{0},z_{0})=(y_{i-1}, z_{i-1}) \;  \mbox{and} \; F_{i}(x_{N},y_{N},z_{N})=(y_{i}, z_{i}) \; .$$
Here $\alpha_{i}, \gamma_{i}$ are free variables with $|\alpha_{i}|< 1$, $|\gamma_{i}|< 1$ and $\beta_{i}$ are constrained variable such that $|\beta_{i}|+ |\gamma_{i}|< 1$. Then the generalized IFS
$$\{I\times \mathbb{R}^{2};\; w_{i}(x,y,z)\; :  \; i=1, \; 2, \; \dotsc,  N\}$$
has an attractor $G$ such that $G=\bigcup_{i=1}^{N}w_{i}(G)=\bigcup_{i=1}^{N}\{w_{i}(x,y,z) \; : \; (x,y,z)\in G\}$ \cite{chand-2007}. The attractor $G$ is the graph of a vector valued function $f:I\to\mathbb{R}^{2}$ such that $f(x_{i})=(y_{i}, z_{i})$ for $i=0$, $1$, $\ldots$, $N$ and $G=\{(x, f(x)): x\in I, f(x)=(y(x), z(x)) \}$. If $f=(f_{1}, f_{2})$, then the projection of the attractor $G$ on $\mathbb{R}^{2}$ is the graph of the function $f_{1}$ which satisfies $f_{1}(x_{i})=y_{i}$ and is of the form
$$f_{1}(L_{i}(x))= F_{i}^{1}(x, f_{1}(x), f_{2}(x))= \alpha_{i}f_{1}(x)+\beta_{i}f_{2}(x)+ c_{i}x+ d_{i}, \; x\in I $$
is known as CHFIF corresponding to the data $\{(x_{i},y_{i}) \in I\times\mathbb{R} \; : \; i=0, \; 1, \; \dotsc, \;  N\}$.
\subsection{Graph-directed Iterated Function Systems}
Let $G = (V, E)$ be a directed graph where $V$ denote the set of vertices and $E$ is the set of edges. For all $u, v \in V$, let $E^{uv}$ denote the
set of edges from $u$ to $v$ with elements $e_{i}^{uv}, i= 1, \; 2, \; \ldots, \;  K^{uv}$ where $K^{uv}$ denotes the number of elements of $E^{uv}$. An iterated function system realizing the graph $G$ is given by a collection of metric spaces $(X^{v}, \rho^{v}), v \in V$, and of contraction mappings  $w_{i}^{uv} : X^{v} \to X^{u}$ corresponding to the edge $e_{i}^{uv}$ in the opposite direction of $e_{i}^{uv}$. An attractor (or invariant list) for such an iterated function system is a list of nonempty compact sets $A^{u} \subset X^{u}$ such that for all $u \in V$ ,
$$ A^{u} = \bigcup_{v\in V}\bigcup_{i=1}^{K^{uv}}w_{i}^{uv}(A^{v}) \; .$$
Then $(X^{u}; w_{i}^{uv})$ is the graph directed iterated function system (GDIFS) realizing the graph $G$ \cite{edgar-2008, mauldin-88}.
\begin{example}
 One can see \cite{deniz-15, demir-10}.
\end{example}

\section{Graph Directed Coalescence FIF}

%
In this section, for a finite number of data sets,  generalized graph-directed iterated function system (GDIFS) is defined so that projection of each attractor on $\mathbb{R}^{2}$ is the graph of a CHFIF which interpolates the corresponding data set and call it as graph-directed coalescence hidden-variable fractal interpolation function. For simplicity, only two sets of data are considered. Let the two data sets as
$$
D^{1}=\{(x_{0}^{1}, y_{0}^{1}), (x_{1}^{1}, y_{1}^{1}), \ldots, (x_{N}^{1}, y_{N}^{1})\}
$$
$$
D^{2}=\{(x_{0}^{2}, y_{0}^{2}), (x_{1}^{2}, y_{1}^{2}), \ldots, (x_{M}^{2}, y_{M}^{2})\}
$$
with $N, M\geq 2$ and
\begin{equation}
\frac{x_{i}^{1}-x_{i-1}^{1}}{x_{M}^{2}-x_{0}^{2}}< 1  \; \mbox{and} \; \frac{x_{j}^{2}-x_{j-1}^{2}}{x_{N}^{1}-x_{0}^{1}}< 1   \label{deqn1}
\end{equation}
for all $i= 1, 2, \ldots, N$ and $j= 1, 2, \ldots, M$. By introducing two set of real parameters $z_{i}^{1}, z_{j}^{2}$ for $i= 1, 2, \ldots, N$ and $j= 1, 2, \ldots, M$, consider the  two generalized data set as
$$
\mathcal{D}^{1}=\{(x_{0}^{1}, y_{0}^{1}, z_{0}^{1}), (x_{1}^{1}, y_{1}^{1}, z_{1}^{1}), \ldots, (x_{N}^{1}, y_{N}^{1}, z_{N}^{1})\}
$$
$$
\mathcal{D}^{2}=\{(x_{0}^{2}, y_{0}^{2}, z_{0}^{2}), (x_{1}^{2}, y_{1}^{2}, z_{1}^{2}), \ldots, (x_{M}^{2}, y_{M}^{2}, z_{M}^{2})\}
\; .
$$
Now consider the directed graph $G=(V,E)$ with $V=\{1, 2\}$ is such that
$$
K^{11}+ K^{12}= N \; \mbox{and} \; K^{21}+ K^{22}= M \; .
$$
give a picture.\\
To construct a generalized GDIFS associated with the data $\mathcal{D}^{r}, (r=1, 2)$ and realizing the graph $G$ consider  the functions $w_{n}^{rs}:\mathbb{R}^{3}\rightarrow\mathbb{R}^{3}$ defined as
$$
w_{n}^{rs}(x,y,z)=(L_{n}^{rs}(x), F_{n}^{rs}(x,y,z)), \; n=1, 2, \ldots, K^{rs}
$$
are such that

\begin{itemize}
\item
 $\left\{\begin{array}{ll} w_{n}^{11}(x_{0}^{1},y_{0}^{1},z_{0}^{1})=(x_{n-1}^{1},y_{n-1}^{1},z_{n-1}^{1}) &   \vspace{-.2cm}\\
   & \mbox{ for } n= 1, 2, \ldots, K^{11} \vspace{-.2cm} \\
 w_{n}^{11}(x_{N}^{1},y_{N}^{1},z_{N}^{1})=(x_{n}^{1},y_{n}^{1},z_{n}^{1}) &  \end{array} \right.$
 \vspace{.2cm}
 \item
 $\left\{\begin{array}{ll} w_{n-K^{11}}^{12}(x_{0}^{2},y_{0}^{2},z_{0}^{2})=(x_{n-1}^{1},y_{n-1}^{1},z_{n-1}^{1}) &   \vspace{-.2cm}\\
   & \mbox{ for } n= K^{11}+1,  \ldots, K^{11}+K^{12}=N \vspace{-.2cm} \\
 w_{n-K^{11}}^{11}(x_{M}^{2},y_{M}^{2},z_{M}^{2})=(x_{n}^{1},y_{n}^{1},z_{n}^{1}) &  \end{array} \right.$
 \vspace{.2cm}
 \item
 $\left\{\begin{array}{ll} w_{n}^{21}(x_{0}^{1},y_{0}^{1},z_{0}^{1})=(x_{n-1}^{2},y_{n-1}^{2},z_{n-1}^{2}) &   \vspace{-.2cm}\\
   & \mbox{ for } n= 1, 2, \ldots, K^{21} \vspace{-.2cm} \\
 w_{n}^{21}(x_{N}^{1},y_{N}^{1},z_{N}^{1})=(x_{n}^{2},y_{n}^{2},z_{n}^{2}) &  \end{array} \right.$
 \vspace{.2cm}
 \item
 $\left\{\begin{array}{ll} w_{n-K^{21}}^{22}(x_{0}^{2},y_{0}^{2},z_{0}^{2})=(x_{n-1}^{2},y_{n-1}^{2},z_{n-1}^{2}) &   \vspace{-.2cm}\\
   & \mbox{ for } n= K^{21}+1,  \ldots, K^{21}+K^{22}=M \vspace{-.2cm} \\
 w_{n-K^{21}}^{22}(x_{M}^{2},y_{M}^{2},z_{M}^{2})=(x_{n}^{2},y_{n}^{2},z_{n}^{2}) &  \end{array} \right.$
\end{itemize}
\vspace{.2cm}
From each of the above conditions, the following can derive respectively.
\begin{equation}
\left\{\begin{array}{ll} a_{n}^{11}x_{0}^{1}+ b_{n}^{11}= x_{n-1}^{1} &   \\
a_{n}^{11}x_{N}^{1}+ b_{n}^{11}= x_{n}^{1} &   \\
c_{n}^{11}x_{0}^{1}+ \alpha_{n}^{11}y_{0}^{1}+ \beta_{n}^{11}z_{0}^{1}+ d_{n}^{11}= y_{n-1}^{1} &   \vspace{-.2cm}\\
   & \mbox{ for } n= 1, 2, \ldots, K^{11} \vspace{-.2cm} \\
 c_{n}^{11}x_{N}^{1}+ \alpha_{n}^{11}y_{N}^{1}+ \beta_{n}^{11}z_{N}^{1}+ d_{n}^{11}= y_{n}^{1} & \\
 e_{n}^{11}x_{0}^{1}+ \gamma_{n}^{11}z_{0}^{1}+ f_{n}^{11}= z_{n-1}^{1} & \\
e_{n}^{11}x_{N}^{1}+ \gamma_{n}^{11}z_{N}^{1}+ f_{n}^{11}= z_{n}^{1} &  \end{array} \right.
\label{eqn1}
\end{equation}
\vspace{.2cm}
\begin{equation}
\left\{\begin{array}{ll} a_{n-K^{11}}^{12}x_{0}^{2}+ b_{n-K^{11}}^{12}= x_{n-1}^{1} &   \\
a_{n-K^{11}}^{12}x_{M}^{2}+ b_{n-K^{11}}^{12}= x_{n}^{1} &   \\
c_{n-K^{11}}^{12}x_{0}^{2}+ \alpha_{n-K^{11}}^{12}y_{0}^{2}+ \beta_{n-K^{11}}^{12}z_{0}^{2}+ d_{n-K^{11}}^{12}= y_{n-1}^{1} &   \vspace{-.2cm}\\
   & \mbox{ for } n= K^{11}+1,  \ldots, N \vspace{-.2cm} \\
 c_{n-K^{11}}^{12}x_{M}^{2}+ \alpha_{n-K^{11}}^{12}y_{M}^{2}+ \beta_{n-K^{11}}^{12}z_{M}^{2}+ d_{n-K^{11}}^{12}= y_{n}^{1} & \\
 e_{n-K^{11}}^{12}x_{0}^{2}+ \gamma_{n-K^{11}}^{12}z_{0}^{2}+ f_{n-K^{11}}^{12}= z_{n-1}^{1} & \\
e_{n-K^{11}}^{12}x_{M}^{2}+ \gamma_{n-K^{11}}^{12}z_{M}^{2}+ f_{n-K^{11}}^{12}= z_{n}^{1} &  \end{array} \right.
\label{eqn2}
\end{equation}
\vspace{.2cm}
\begin{equation}
\left\{\begin{array}{ll} a_{n}^{21}x_{0}^{1}+ b_{n}^{21}= x_{n-1}^{2} &   \\
a_{n}^{21}x_{N}^{1}+ b_{n}^{21}= x_{n}^{2} &   \\
c_{n}^{21}x_{0}^{1}+ \alpha_{n}^{21}y_{0}^{1}+ \beta_{n}^{21}z_{0}^{1}+ d_{n}^{21}= y_{n-1}^{2} &   \vspace{-.2cm}\\
   & \mbox{ for } n= 1, 2, \ldots, K^{21} \vspace{-.2cm} \\
 c_{n}^{21}x_{N}^{1}+ \alpha_{n}^{21}y_{N}^{1}+ \beta_{n}^{21}z_{N}^{1}+ d_{n}^{21}= y_{n}^{2} & \\
 e_{n}^{21}x_{0}^{1}+ \gamma_{n}^{21}z_{0}^{1}+ f_{n}^{21}= z_{n-1}^{2} & \\
e_{n}^{21}x_{N}^{1}+ \gamma_{n}^{21}z_{N}^{1}+ f_{n}^{21}= z_{n}^{2} &  \end{array} \right.
\label{eqn3}
\end{equation}
\vspace{.2cm}
\begin{equation}
\left\{\begin{array}{ll} a_{n-K^{21}}^{22}x_{0}^{2}+ b_{n-K^{21}}^{22}= x_{n-1}^{2} &   \\
a_{n-K^{21}}^{22}x_{M}^{2}+ b_{n-K^{21}}^{22}= x_{n}^{2} &   \\
c_{n-K^{21}}^{22}x_{0}^{2}+ \alpha_{n-K^{21}}^{22}y_{0}^{2}+ \beta_{n-K^{21}}^{22}z_{0}^{2}+ d_{n-K^{21}}^{22}= y_{n-1}^{2} &   \vspace{-.2cm}\\
   & \mbox{ for } n= K^{21}+1,  \ldots, M \vspace{-.2cm} \\
 c_{n-K^{21}}^{22}x_{M}^{2}+ \alpha_{n-K^{21}}^{22}y_{M}^{2}+ \beta_{n-K^{21}}^{22}z_{M}^{2}+ d_{n-K^{21}}^{22}= y_{n}^{2} & \\
 e_{n-K^{21}}^{22}x_{0}^{2}+ \gamma_{n-K^{21}}^{22}z_{0}^{2}+ f_{n-K^{21}}^{22}= z_{n-1}^{2} & \\
e_{n-K^{21}}^{22}x_{M}^{2}+ \gamma_{n-K^{21}}^{22}z_{M}^{2}+ f_{n-K^{21}}^{22}= z_{n}^{2} &  \end{array} \right.
\label{eqn4}
\end{equation}
From the linear system of equations (\ref{eqn1}), (\ref{eqn2}), (\ref{eqn3}) and (\ref{eqn4}) the constants $a_{i}^{rs}$, $b_{i}^{rs}$, $c_{i}^{rs}$, $d_{i}^{rs}$, $e_{i}^{rs}$ and $f_{i}^{rs}$ for $r, s\in \{1, 2\}$, $i= 1, 2, \ldots, K^{rs}$ are determined as follows
\begin{center}
\begin{tabular}{cc}
$a_{n}^{11}=\frac{x_{n}^{1}-x_{n-1}^{1}}{x_{N}^{1}-x_{0}^{1}}$  & $a_{n}^{12}=\frac{x_{n}^{1}-x_{n-1}^{1}}{x_{M}^{2}-x_{0}^{2}}$\\
$b_{n}^{11}=\frac{x_{N}^{1}x_{n-1}^{1}-x_{0}^{1}x_{n}^{1}}{x_{N}^{1}-x_{0}^{1}}$  & $b_{n}^{12}=\frac{x_{M}^{2}x_{n-1}^{1}-x_{0}^{2}x_{n}^{1}}{x_{M}^{2}-x_{0}^{2}}$\\
$c_{n}^{11}=\frac{y_{n}^{1}-y_{n-1}^{1}-\alpha_{n}^{11}(y_{N}^{1}-y_{0}^{1})-\beta_{n}^{11}(z_{N}^{1}-z_{0}^{1})}{x_{N}^{1}-x_{0}^{1}}$ &
$c_{n}^{12}=\frac{y_{n}^{1}-y_{n-1}^{1}-\alpha_{n}^{12}(y_{M}^{2}-y_{0}^{2})-\beta_{n}^{12}(z_{M}^{2}-z_{0}^{2})}{x_{M}^{2}-x_{0}^{2}}$\\
$d_{n}^{11}=\frac{x_{N}^{1}y_{n-1}^{1}-x_{0}^{1}y_{n}^{1}-\alpha_{n}^{11}(x_{N}^{1}y_{0}^{1}-x_{0}^{1}y_{N}^{1})-\beta_{n}^{11}(x_{N}^{1}z_{0}^{1}-x_{0}^{1}z_{N}^{1})}{x_{N}^{1}-x_{0}^{1}}$ & $d_{n}^{12}=\frac{x_{M}^{2}y_{n-1}^{1}-x_{0}^{2}y_{n}^{1}-\alpha_{n}^{12}(x_{M}^{2}y_{0}^{2}-x_{0}^{2}y_{M}^{2})-\beta_{n}^{12}(x_{M}^{2}z_{0}^{2}-x_{0}^{2}z_{M}^{2})}{x_{M}^{2}-x_{0}^{2}}$\\
$e_{n}^{11}=\frac{z_{n}^{1}-z_{n-1}^{1}-\gamma_{n}^{11}(z_{N}^{1}-z_{0}^{1})}{x_{N}^{1}-x_{0}^{1}}$ &
$e_{n}^{12}=\frac{z_{n}^{1}-z_{n-1}^{1}-\gamma_{n}^{12}(z_{M}^{2}-z_{0}^{2})}{x_{M}^{2}-x_{0}^{2}}$\\
$f_{n}^{11}=\frac{x_{N}^{1}z_{n-1}^{1}-x_{0}^{1}z_{n}^{1}-\gamma_{n}^{11}(x_{N}^{1}z_{0}^{1}-x_{0}^{1}z_{N}^{1})}{x_{N}^{1}-x_{0}^{1}}$ &
$f_{n}^{12}=\frac{x_{M}^{2}z_{n-1}^{1}-x_{0}^{2}z_{n}^{1}-\gamma_{n}^{12}(x_{M}^{2}z_{0}^{2}-x_{0}^{2}z_{M}^{2})}{x_{M}^{2}-x_{0}^{2}}$
\end{tabular}
\end{center}
\begin{center}
\begin{tabular}{cc}
$a_{n}^{21}=\frac{x_{n}^{2}-x_{n-1}^{2}}{x_{N}^{1}-x_{0}^{1}}$  & $a_{n}^{22}=\frac{x_{n}^{2}-x_{n-1}^{2}}{x_{M}^{2}-x_{0}^{2}}$\\
$b_{n}^{21}=\frac{x_{N}^{1}x_{n-1}^{2}-x_{0}^{1}x_{n}^{2}}{x_{N}^{1}-x_{0}^{1}}$  & $b_{n}^{22}=\frac{x_{M}^{2}x_{n-1}^{2}-x_{0}^{2}x_{n}^{2}}{x_{M}^{2}-x_{0}^{2}}$\\
$c_{n}^{21}=\frac{y_{n}^{2}-y_{n-1}^{2}-\alpha_{n}^{21}(y_{N}^{1}-y_{0}^{1})-\beta_{n}^{21}(z_{N}^{1}-z_{0}^{1})}{x_{N}^{1}-x_{0}^{1}}$ &
$c_{n}^{22}=\frac{y_{n}^{2}-y_{n-1}^{2}-\alpha_{n}^{22}(y_{M}^{2}-y_{0}^{2})-\beta_{n}^{22}(z_{M}^{2}-z_{0}^{2})}{x_{M}^{2}-x_{0}^{2}}$\\
$d_{n}^{21}=\frac{x_{N}^{1}y_{n-1}^{2}-x_{0}^{1}y_{n}^{2}-\alpha_{n}^{21}(x_{N}^{1}y_{0}^{1}-x_{0}^{1}y_{N}^{1})-\beta_{n}^{21}(x_{N}^{1}z_{0}^{1}-x_{0}^{1}z_{N}^{1})}{x_{N}^{1}-x_{0}^{1}}$ & $d_{n}^{22}=\frac{x_{M}^{2}y_{n-1}^{2}-x_{0}^{2}y_{n}^{2}-\alpha_{n}^{22}(x_{M}^{2}y_{0}^{2}-x_{0}^{2}y_{M}^{2})-\beta_{n}^{22}(x_{M}^{2}z_{0}^{2}-x_{0}^{2}z_{M}^{2})}{x_{M}^{2}-x_{0}^{2}}$\\
$e_{n}^{21}=\frac{z_{n}^{2}-z_{n-1}^{2}-\gamma_{n}^{21}(z_{N}^{1}-z_{0}^{1})}{x_{N}^{1}-x_{0}^{1}}$ &
$e_{n}^{22}=\frac{z_{n}^{2}-z_{n-1}^{2}-\gamma_{n}^{22}(z_{M}^{2}-z_{0}^{2})}{x_{M}^{2}-x_{0}^{2}}$\\
$f_{n}^{21}=\frac{x_{N}^{1}z_{n-1}^{2}-x_{0}^{1}z_{n}^{2}-\gamma_{n}^{21}(x_{N}^{1}z_{0}^{1}-x_{0}^{1}z_{N}^{1})}{x_{N}^{1}-x_{0}^{1}}$ &
$f_{n}^{22}=\frac{x_{M}^{2}z_{n-1}^{2}-x_{0}^{2}z_{n}^{2}-\gamma_{n}^{22}(x_{M}^{2}z_{0}^{2}-x_{0}^{2}z_{M}^{2})}{x_{M}^{2}-x_{0}^{2}}$
\end{tabular}
\end{center}
The following theorem shows that each maps $w_{n}^{rs}$ is contraction with respect to metric equivalent to the Euclidean metric and ensures the existence  of attractors of generalized GDIFS.
\begin{thm}\label{thm1}
Let $\{\mathbb{R}^{3}; w_{n}^{rs}, n=1, 2, \ldots, K^{rs} \}$ be the generalized GDIFS defined above realizing the graph and associated with the data sets $\mathcal{D}^{r}, (r= 1, 2)$  which satisfy  (\ref{deqn1}). If $|\alpha_{n}^{rs}|< 1, |\gamma_{n}^{rs}|< 1$ and $\beta_{n}^{rs}$ is chosen such that $|\beta_{n}^{rs}|+|\gamma_{n}^{rs}|< 1$ for all $r, s \in\{1, 2\}$ and $n=1, 2, \ldots, K^{rs}$. Then there exists a metric $\delta$ on $\mathbb{R}^{3}$ equivalent to the Euclidean metric, such that the GDIFS is hyperbolic with respect to $\delta$. In particular, there exists   non empty compact sets $G^{r}$ such that
$$ G^{r}=\bigcup_{s=1}^{2}\bigcup_{n=1}^{K^{rs}}w_{n}^{rs}(G^{s}) \; .$$
\end{thm}
\begin{proof}
Proof follows in the similar line of Theorem 2.1.1, \cite{chand-04} and using above condition (\ref{deqn1}).
\end{proof}
Following is the main result regarding existence of coalescence Hidden-variable FIFs for generalized GDIFS.
\begin{thm}
Let $G^{r}, r\in V$ be the attractors of the generalized GDIFS as in Theorem~ \ref{thm1}. Then $G^{r}, r\in V$ is the graph of a vector valued continuous function $f^{r}: I^{r}\to \mathbb{R}^{2}$ such that for $r\in V$, $f^{r}(x_{n}^{r})= (y_{n}^{r}, z_{n}^{r})$ for all $n=1, 2, \ldots, N^{r}$. If $f^{r}= (f_{1}^{r}, f_{2}^{r})$ then the projection of the attractors $G^{r}, r\in V$ on $\mathbb{R}^{2}$ is the graph of the continuous function $f_{1}^{r}: I^{r}\to \mathbb{R}$ known as CHFIF such that for $r\in V$, $f^{r}(x_{n}^{r})= (y_{n}^{r}) $. That is $G^{r}|_{\mathbb{R}^{2}}=\{(x, f_{1}^{r}(x)): x\in I^{r} \}$
\end{thm}
\begin{proof}
Consider the vector valued function spaces
$$
\mathcal{F}=\{f:[x_{0}^{1}, x_{N}^{1}]\rightarrow \mathbb{R}^{2} \; \mbox{continuous such that} \; f(x_{0}^{1})= (y_{0}^{1}, z_{0}^{1}), f(x_{N}^{1})= (y_{N}^{1}, z_{N}^{1}) \}
$$
$$
\mathcal{H}=\{h:[x_{0}^{2}, x_{M}^{2}]\rightarrow \mathbb{R}^{2} \; \mbox{continuous such that} \; h(x_{0}^{2})= (y_{0}^{2}, z_{0}^{2}), h(x_{M}^{2})= (y_{M}^{2}, z_{M}^{2}) \}
$$
with metrics
$$
d_{\mathcal{F}}(f_{1}, f_{2})= \sup_{x\in [x_{0}^{1}, x_{N}^{1}]} \|f_{1}(x)-f_{2}(x)\|
$$
$$
d_{\mathcal{H}}(h_{1}, h_{2})= \sup_{x\in [x_{0}^{2}, x_{M}^{2}]} \|h_{1}(x)- h_{2}(x)\|
$$
respectively, where $\|.\|$ denotes a norm on $\mathbb{R}^{2}$. Since $(\mathcal{F}, d_{\mathcal{F}})$ and $(\mathcal{H}, d_{\mathcal{H}})$ are complete metric spaces, then $(\mathcal{F}\times \mathcal{H}, d)$ is also a complete metric space where
$$
d((f_{1}, h_{1}), (f_{2}, h_{2}))= \max\{d_{\mathcal{F}}(f_{1}, f_{2}), d_{\mathcal{H}}(h_{1}, h_{2})\}
\; .
$$
Following are the affine maps.
$$I: [x_{0}^{1}, x_{N}^{1}]\to [x_{n-1}^{1}, x_{n}^{1}], I_{n}^{1}(x)= a_{n}^{11}x+ b_{n}^{11} \; \mbox{for} \;  n=1, 2, \ldots, K^{11}$$
$$I: [x_{0}^{2}, x_{M}^{2}]\to [x_{n-1}^{1}, x_{n}^{1}], I_{n}^{1}(x)= a_{n-K^{11}}^{12}x+ b_{n-K^{11}}^{12} \;  \mbox{for} \; n=K^{11}+1,  \ldots, K^{12}$$
$$J: [x_{0}^{1}, x_{N}^{1}]\to [x_{n-1}^{2}, x_{n}^{2}], I_{n}^{2}(x)= a_{n}^{21}x+ b_{n}^{21} \; \mbox{for} \;  n=1, 2, \ldots, K^{21}$$
$$J: [x_{0}^{2}, x_{M}^{2}]\to [x_{n-1}^{2}, x_{n}^{2}], I_{n}^{2}(x)= a_{n-K^{21}}^{22}x+ b_{n-K^{21}}^{22} \; \mbox{for} \; n=K^{21}+1,  \ldots, K^{22}$$
Now define the mapping
$$T: \mathcal{F}\times\mathcal{H}\to\mathcal{F}\times\mathcal{H}$$
$$T(f, h)(x, y)= (\widetilde{f}(x), \widetilde{h}(y))$$
where for $x\in[x_{n-1}^{1}, x_{n}^{1}]$\\
\begin{multline}
\widetilde{f}(x)= \bigg\{ \begin{matrix}{}
(c_{n}^{11}I_{n}^{-1}(x) + \alpha_{n}^{11}y_{f}^{1}(I_{n}^{-1}(x))+ \beta_{n}^{11}z_{f}^{1}(I_{n}^{-1}(x))+ d_{n}^{11}, &&\\
\gamma_{n}^{11}z_{f}^{1}(I_{n}^{-1}(x))+ e_{n}^{11}I_{n}^{-1}(x)+ f_{n}^{11}) &  \mathrm{for} \; n=1, 2, \ldots, K^{11}\\
(c_{n-K^{11}}^{12}I_{n}^{-1}(x) + \alpha_{n-K^{11}}^{12}y_{h}^{2}(I_{n}^{-1}(x))+ \beta_{n-K^{11}}^{12}z_{h}^{2}(I_{n}^{-1}(x))+ d_{n-K^{11}}^{12}, &&\\ \gamma_{n-K^{11}}^{12}z_{h}^{2}(I_{n}^{-1}(x))+ e_{n-K^{11}}^{12}I_{n}^{-1}(x)+ f_{n-K^{11}}^{12}) & \mathrm{for} \; n=K^{11}+ 1,  \ldots, N \; .
\end{matrix}
\end{multline}
and for $x\in[x_{m-1}^{2}, x_{m}^{2}]$\\
\begin{multline}
\widetilde{h}(x)= \bigg\{ \begin{matrix}{}
(c_{m}^{21}J_{n}^{-1}(x) + \alpha_{m}^{21}y_{f}^{1}(J_{m}^{-1}(x))+ \beta_{m}^{21}z_{f}^{1}(J_{m}^{-1}(x))+ d_{m}^{21}, &&\\
\gamma_{m}^{21}z_{f}^{1}(I_{m}^{-1}(x))+ e_{m}^{21}J_{m}^{-1}(x)+ f_{m}^{21}) & \mathrm{for} \; m=1, 2, \ldots, K^{21}\\
(c_{m-K^{21}}^{22}J_{m}^{-1}(x) + \alpha_{m-K^{21}}^{22}y_{h}^{2}(J_{m}^{-1}(x))+ \beta_{m-K^{21}}^{22}z_{h}^{2}(J_{m}^{-1}(x))+ d_{m-K^{21}}^{22}, &&\\ \gamma_{m-K^{21}}^{22}z_{h}^{2}(J_{m}^{-1}(x))+ e_{m-K^{21}}^{22}I_{m}^{-1}(x)+ f_{m-K^{21}}^{22}) & \mathrm{for} \; m=K^{21}+ 1,  \ldots, M \; .
\end{matrix}
\end{multline}
Now using equations $(\ref{eqn1})- (\ref{eqn4})$ it is clear that,
$$\widetilde{f}(x_{0}^{1})= F_{1}(I_{n}^{-1}(x), y_{f}^{1}(I_{n}^{-1}(x)), z_{f}^{1}(I_{n}^{-1}(x)))= (y_{0}^{1}, z_{0}^{1})$$
$$\widetilde{f}(x_{N}^{1})= F_{N}(I_{n}^{-1}(x), y_{h}^{2}(I_{n}^{-1}(x)), z_{h}^{2}(I_{n}^{-1}(x)))= (y_{N}^{1}, z_{N}^{1}) \; .$$
Similarly, $\widetilde{h}(x_{0}^{2})= (y_{0}^{2}, z_{0}^{2}) $, $\widetilde{h}(x_{M}^{2})= (y_{M}^{2}, z_{M}^{2}) $. Which proves that $T$ maps $\mathcal{F}\times\mathcal{H}$ into itself. Since for each $n=1, 2, \ldots, N$, $I_{n}^{-1}(x)$ is continuous and therefore, $\widetilde{f}$ is continuous on each subintervals $[x_{n-1}^{1}, x_{n}^{1}]$.\\
For $n=1, 2, \ldots, K^{11}$, using (\ref{eqn1}) it follows that $\widetilde{f}(x_{n}^{1-})=\widetilde{f}(x_{n}^{1+})= (y_{n}^{1}, z_{n}^{1})$.\\
For $n=K^{11}+1,  \ldots, N-1$, using (\ref{eqn2}) it follows that $\widetilde{f}(x_{n}^{1-})=\widetilde{f}(x_{n}^{1+})= (y_{n}^{1}, z_{n}^{1})$.\\
For $n=K^{11}$, using (\ref{eqn1}) and (\ref{eqn2}) it follows that $\widetilde{f}(x_{n}^{1-})=\widetilde{f}(x_{n}^{1+})= (y_{n}^{1}, z_{n}^{1})$ since $I_{n}^{-1}(x_{n}^{1})= x_{N}^{1}$ and $I_{n+1}^{-1}(x_{n}^{1})= x_{0}^{2}$.\\
Hence $\widetilde{f}$ is continuous on $I$. Similarly it can be shown that $\widetilde{h}$ is continuous on $J$. Consequently $T$ is continuous. \\
To show that $T$ is a contraction map on $\mathcal{F}\times\mathcal{H}$, let $T(f_{1}, f_{2})=(\widetilde{f}_{1}, \widetilde{f}_{2})$ and $T(h_{1}, h_{2})=(\widetilde{h}_{1}, \widetilde{h}_{2})$. Now
\begin{align*}
\sup_{x\in[x_{0}^{1}, x_{K^{11}}^{1}]}\{\|\widetilde{f}_{1}(x)-\widetilde{f}_{2}(x)\|\}  &=
\max_{\substack{n=1, 2, \ldots, K^{11} \\ x\in[x_{n-1}^{1}, x_{n}^{1}]}}\{\|\alpha_{n}^{11}(y_{f_{1}}^{1}(I_{n}^{-1}(x))- y_{f_{2}}^{1}(I_{n}^{-1}(x)))\\ &+ \beta_{n}^{11}(z_{f_{1}}^{1}(I_{n}^{-1}(x))- z_{f_{2}}^{1}(I_{n}^{-1}(x))), \gamma_{n}^{11}(z_{f_{1}}^{1}(I_{n}^{-1}(x))- z_{f_{2}}^{1}(I_{n}^{-1}(x)))\|  \}\\
&\leq \delta^{11}\max_{\substack{n=1, 2, \ldots, K^{11} \\ x\in[x_{n-1}^{1}, x_{n}^{1}]}}\{y_{f_{1}}^{1}(I_{n}^{-1}(x))- y_{f_{2}}^{1}(I_{n}^{-1}(x))+ z_{f_{1}}^{1}(I_{n}^{-1}(x))- z_{f_{2}}^{1}(I_{n}^{-1}(x)), \\
& z_{f_{1}}^{1}(I_{n}^{-1}(x))- z_{f_{2}}^{1}(I_{n}^{-1}(x))\}\\
&\leq \delta^{11} d_{\mathcal{F}}(f_{1}, f_{2}) .
\end{align*}
\begin{align*}
\sup_{x\in[x_{K^{11}}^{1}, x_{N}^{1}]}\{\|\widetilde{f}_{1}(x)-\widetilde{f}_{2}(x)\|\}  &=
\max_{\substack{n=K^{11}+1, \ldots, N \\ x\in[x_{n-1}^{1}, x_{n}^{1}]}}\{\|\alpha_{n-K^{11}}^{12}(y_{h_{1}}^{2}(I_{n}^{-1}(x))- y_{h_{2}}^{2}(I_{n}^{-1}(x)))\\ &+ \beta_{n-K^{11}}^{12}(z_{h_{1}}^{2}(I_{n}^{-1}(x))- z_{h_{2}}^{2}(I_{n}^{-1}(x))), \gamma_{n-K^{11}}^{12}(z_{h_{1}}^{2}(I_{n}^{-1}(x))- z_{h_{2}}^{2}(I_{n}^{-1}(x)))\|  \}\\
&\leq \delta^{12}\max_{\substack{n=K^{11}+1, \ldots, N \\ x\in[x_{n-1}^{1}, x_{n}^{1}]}}\{y_{h_{1}}^{2}(I_{n}^{-1}(x))- y_{h_{2}}^{2}(I_{n}^{-1}(x))+ z_{h_{1}}^{2}(I_{n}^{-1}(x))- z_{h_{2}}^{2}(I_{n}^{-1}(x)), \\
& z_{h_{1}}^{2}(I_{n}^{-1}(x))- z_{h_{2}}^{2}(I_{n}^{-1}(x))\}\\
&\leq \delta^{12} d_{\mathcal{H}}(h_{1}, h_{2}) .
\end{align*}
where $\delta^{11}= \max_{n=1, 2, \ldots, K^{11}} \{|\alpha_{n}^{11}|, |\beta_{n}^{11}|, |\gamma_{n}^{11}|\} < 1$ and $\delta^{12}= \max_{n=K^{11}+1, \ldots, N} \{|\alpha_{n}^{12}|, |\beta_{n}^{12}|, |\gamma_{n}^{12}|\} < 1$. Therefore
$$
d_{\mathcal{F}}(\widetilde{f}_{1}, \widetilde{f}_{2})\leq \max\{\delta^{11}, \delta^{12}\}\max\{d_{\mathcal{F}}(f_{1}, f_{2}), d_{\mathcal{H}}(h_{1}, h_{2})\}
\; .
$$
Similarly, one can have
$$
d_{\mathcal{H}}(\widetilde{h}_{1}, \widetilde{h}_{2})\leq \max\{\delta^{21}, \delta^{22}\}\max\{d_{\mathcal{F}}(f_{1}, f_{2}), d_{\mathcal{H}}(h_{1}, h_{2})\}
\; .
$$
where $\delta^{21}= \max_{n=1, 2, \ldots, K^{21}} \{|\alpha_{n}^{21}|, |\beta_{n}^{21}|, |\gamma_{n}^{21}|\} < 1$ and $\delta^{22}= \max_{n=K^{21}+1, \ldots, M} \{|\alpha_{n}^{22}|, |\beta_{n}^{22}|, |\gamma_{n}^{22}|\} < 1$. Hence
$$d(T(f_{1}, h_{1}), T(f_{2}, h_{2}))= \max\{d_{\mathcal{F}}(\widetilde{f}_{1}, \widetilde{f}_{2}), d_{\mathcal{H}}(\widetilde{h}_{1}, \widetilde{h}_{2})\}\leq \delta \max\{d_{\mathcal{F}}(f_{1}, f_{2}), d_{\mathcal{H}}(h_{1}, h_{2})\}$$
where $\delta= \max\{\delta^{11}, \delta^{12}, \delta^{21}, \delta^{22}\}< 1.$ Which proves that $T$ is contraction mapping. Then by Banach fixed point theorem, $T$ posses a unique fixed point, say $(f_{0}, h_{0})$. \\
Now, for $ n=1, 2, \ldots, K^{11}$
\begin{align*}
f_{0}(x_{n}^{1}) &= (c_{n+1}^{11}I_{n+1}^{-1}(x_{n}^{1}) + \alpha_{n+1}^{11}y_{f_{0}}^{1}(I_{n+1}^{-1}(x_{n}^{1}))+ \beta_{n+1}^{11}z_{f_{0}}^{1}(I_{n+1}^{-1}(x_{n}^{1}))+ d_{n+1}^{11}, \\
&\gamma_{n+1}^{11}z_{f_{0}}^{1}(I_{n+1}^{-1}(x_{n}^{1}))+ e_{n+1}^{11}I_{n+1}^{-1}(x_{n}^{1})+ f_{n+1}^{11}) \\
& =(y_{n}^{1}, z_{n}^{1}) .
\end{align*}
For $ n=K^{11}+1,  \ldots, N-1$
\begin{align*}
f_{0}(x_{n}^{1}) &= (c_{n+1-K^{11}}^{12}I_{n+1}^{-1}(x_{n}^{1}) + \alpha_{n+1-K^{11}}^{12}y_{h_{0}}^{2}(I_{n+1}^{-1}(x_{n}^{1}))+ \beta_{n+1-K^{11}}^{12}z_{h_{0}}^{2}(I_{n+1}^{-1}(x_{n}^{1}))+ d_{n+1-K^{11}}^{12}, \\
&\gamma_{n+1-K^{11}}^{12}z_{h_{0}}^{2}(I_{n+1}^{-1}(x_{n}^{1}))+ e_{n+1-K^{11}}^{12}I_{n+1}^{-1}(x_{n}^{1})+ f_{n+1-K^{11}}^{12}) \\
& =(y_{n}^{1}, z_{n}^{1})
\end{align*}
This shows that $f_{0}$ is the function which interpolates the data $\{(x_{n}^{1}, y_{n}^{1}, z_{n}^{1}) : n=0, 1, \ldots, N\}$. Similarly, it can be shown that $g_{0}$ is the function which interpolates the data $\{(x_{n}^{2}, y_{n}^{2}, z_{n}^{2}) : n=0, 1, \ldots, M\}$.
Now for $x\in[x_{0}^{1}, x_{N}^{1}]$ and $x\in[x_{0}^{2}, x_{M}^{2}]$
\begin{align*}
f_{0}(I_{n}(x))= (c_{n}^{11}x + \alpha_{n}^{11}y_{f_{0}}^{1}(x)+ \beta_{n}^{11}z_{f_{0}}^{1}(x)+ d_{n}^{11}, &&\\
\gamma_{n}^{11}z_{f_{0}}^{1}(x)+ e_{n}^{11}x+ f_{n}^{11}) && \; \mathrm{for} \; n=1, 2, \ldots, K^{11}
\end{align*}
\begin{align*}
f_{0}(I_{n}(x))= (c_{n}^{12}x + \alpha_{n}^{12}y_{h_{0}}^{2}(x)+ \beta_{n}^{12}z_{h_{0}}^{2}(x)+ d_{n}^{12}, &&\\
\gamma_{n}^{12}z_{h_{0}}^{2}(x)+ e_{n}^{12}x+ f_{n}^{12}) && \; \mathrm{for} \; n=1, 2, \ldots, K^{12}
\end{align*}
and
\begin{align*}
h_{0}(J_{n}(x))= (c_{n}^{21}x + \alpha_{n}^{21}y_{f_{0}}^{1}(x)+ \beta_{n}^{21}z_{f_{0}}^{1}(x)+ d_{n}^{21}, &&\\
\gamma_{n}^{21}z_{f_{0}}^{1}(x)+ e_{n}^{21}x+ f_{n}^{21}) && \; \mathrm{for} \; n=1, 2, \ldots, K^{21}
\end{align*}
\begin{align*}
h_{0}(J_{n}(x))= (c_{n}^{22}x + \alpha_{n}^{22}y_{h_{0}}^{2}(x)+ \beta_{n}^{22}z_{h_{0}}^{2}(x)+ d_{n}^{22}, &&\\
\gamma_{n}^{22}z_{h_{0}}^{2}(x)+ e_{n}^{22}x+ f_{n}^{22}) && \; \mathrm{for} \; n=1, 2, \ldots, K^{22}
.
\end{align*}
If $F$ and $H$ are the graphs of $f_{0}$ and $h_{0}$ respectively, then
$$
F=\bigcup_{i=1}^{K^{11}}w_{i}^{11}(F)\bigcup \bigcup_{i=1}^{K^{12}}w_{i}^{12}(H)
$$
$$
H=\bigcup_{i=1}^{K^{21}}w_{i}^{21}(F)\bigcup \bigcup_{i=1}^{K^{22}}w_{i}^{22}(H)
\; .
$$
Now the uniqueness of the attractor imply that $F=G^{1}$ and $H=G^{2}$. That is $G^{1}= \{(x, f_{0}(x)): x\in I\}$ and $G^{2}= \{(x, h_{0}(x)): x\in J\}$.
\end{proof}
\begin{example}
Consider the data sets as
$$D^{1}=\{(0,5), (1,4), (2,1), (3,1), (4,4), (5,5)\}$$
$$D^{2}=\{(0,1), (1,2), (2,3), (3,2), (4,1)\}$$
realizing the graph with $K^{11}=3$, $K^{12}=2$, $K^{21}=1$, $K^{22}=3$.
Take the generalized data set $$\mathcal{D}^{1}=\{(0,5,5), (1,4,4), (2,1,1), (3,1,1), (4,4,4), (5,5,5)\}$$ and
$$\mathcal{D}^{2}=\{(0,1,1), (1,2,2), (2,3,3), (3,2,2), (4,1,1)\}$$
corresponding to $D^{1}$ and $D^{2}$ respectively. Here $y_{n}=z_{n}$ for both the generalized data sets. Choose $\alpha_{n}^{rs}=1/3$, $\beta_{n}^{rs}=1/3$, $\gamma_{n}^{rs}=1/3$ for all $r,s \in\{1, 2\}$ and $n=1, 2, \ldots, K^{rs}$.
Then Fig~$\ref{figure1}$ and Fig~$\ref{figure2}$ are the attractors of the corresponding generalized  GDIFS. \\
Keeping the free variables and constrained variables same, Fig~\ref{figure3} and Fig~\ref{figure4} are the attractors of the generalized GDIFS associated with the generalized data sets
$$\mathcal{D}^{1}=\{(0,5,3), (1,4,2), (2,1,5), (3,1,2), (4,4,1), (5,5,4)\}$$
$$\mathcal{D}^{2}=\{(0,1,2), (1,2,5), (2,3,1), (3,2,3), (4,1,1)\} \; .$$

%
%
%
%
%
%
%
%
%
%
%

Take the generalized data set $$\mathcal{D}^{1}=\{(0,5,3), (1,4,2), (2,1,5), (3,1,2), (4,4,1), (5,5,4)\}$$ and
$$\mathcal{D}^{2}=\{(0,1,2), (1,5,5), (2,3,1), (3,2,3), (4,4,1)\}$$
corresponding to $D^{1}$ and $D^{2}$ respectively. Then Fig~$\ref{figure5}$ and Fig~$\ref{figure6}$ are the attractors of the generalized  GDIFS  with the free variables and constraints variables given in following table \ref{table2}.

\begin{table}[!ht]
\centering
\caption{\label{table2}\it{}}
\begin{tabular}{||c|ccccccccc||}
\hline
$\alpha$ & $\alpha_{1}^{11}$ & $\alpha_{2}^{11}$ & $\alpha_{3}^{11}$ & $\alpha_{1}^{12}$ & $\alpha_{2}^{12}$ & $\alpha_{1}^{21}$ & $\alpha_{1}^{22}$ & $\alpha_{2}^{22}$ & $\alpha_{3}^{22}$\\

\hline

& 0.8 & 0.7 & 0.8 & 0.7 & 0.8 & 0.99 & 0.99 & 0.99 & 0.99\\

\hline

$\beta$ & $\beta_{1}^{11}$ & $\beta_{2}^{11}$ & $\beta_{3}^{11}$ & $\beta_{1}^{12}$ & $\beta_{2}^{12}$ & $\beta_{1}^{21}$ & $\beta_{1}^{22}$ & $\beta_{2}^{22}$ & $\beta_{3}^{22}$\\

\hline

& -0.3 & -0.4 & -0.2 & -0.3 & -0.4 & 0.99 & 0.99 & 0.99 & 0.99\\

\hline

$\gamma$ & $\gamma_{1}^{11}$ & $\gamma_{2}^{11}$ & $\gamma_{3}^{11}$ & $\gamma_{1}^{12}$ & $\gamma_{2}^{12}$ & $\gamma_{1}^{21}$ & $\gamma_{1}^{22}$ & $\gamma_{2}^{22}$ & $\gamma_{3}^{22}$\\

\hline

& 0.5 & 0.3 & 0.6 & 0.5 & 0.3 & 0.005 & 0.005 & 0.005 & 0.005\\

\hline
\end{tabular}
\end{table}
\end{example}

\begin{figure}[!ht]
\begin{center}
\includegraphics[width=7cm]{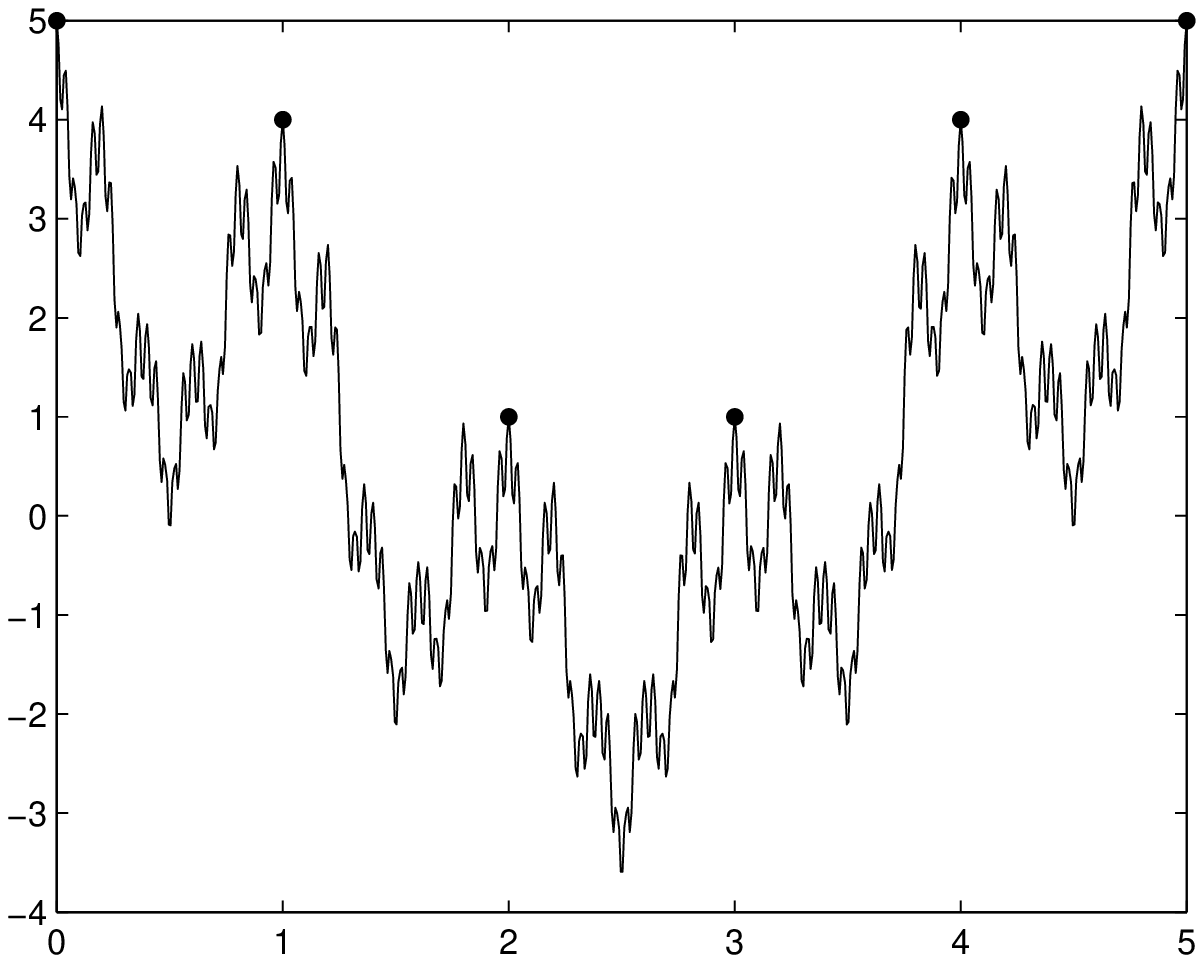}
\caption{}
\label{figure1}
\quad
\includegraphics[width=7cm]{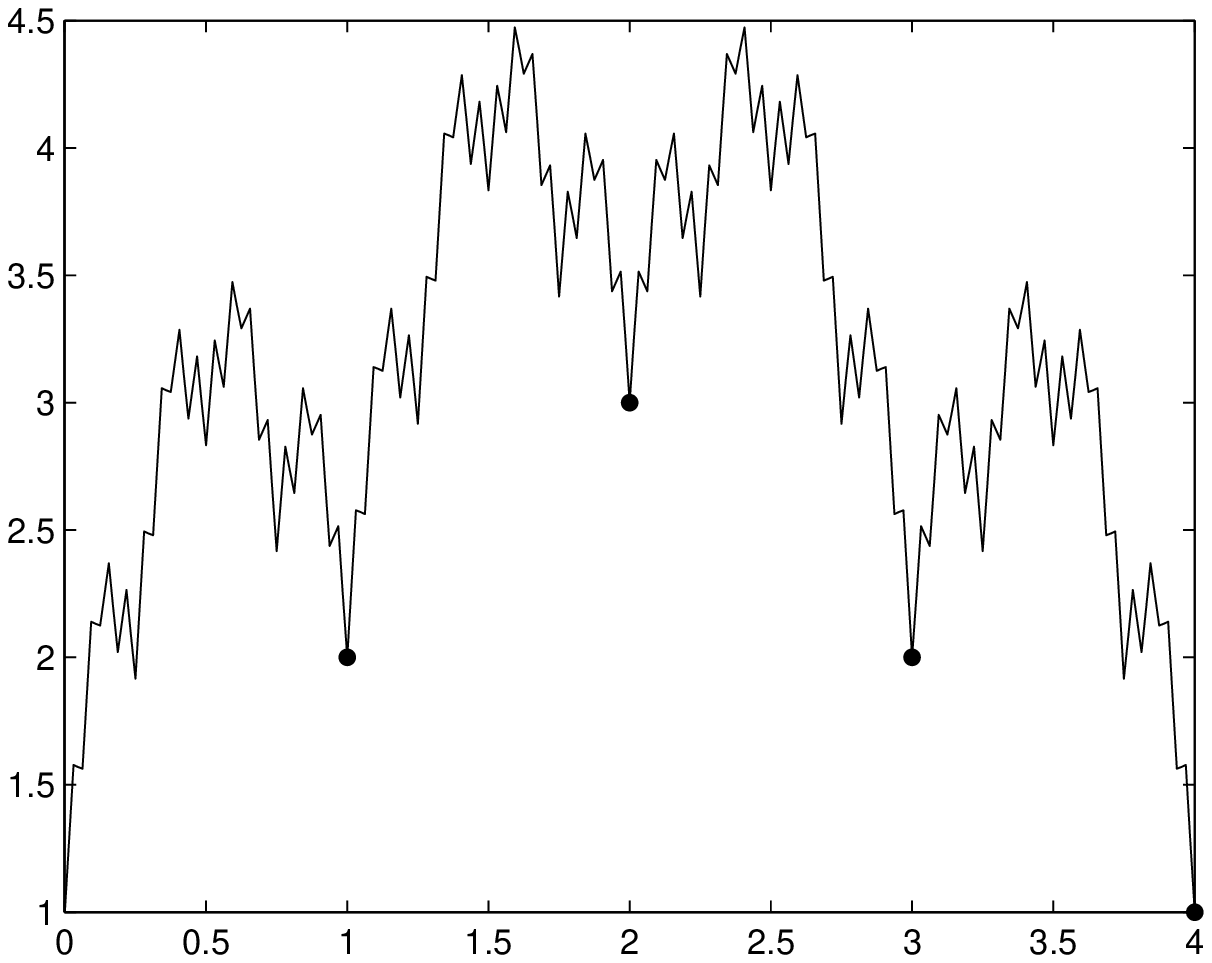}
\caption{}
\label{figure2}
\end{center}
\end{figure}

\begin{figure}[!ht]
\begin{center}
\includegraphics[width=7cm]{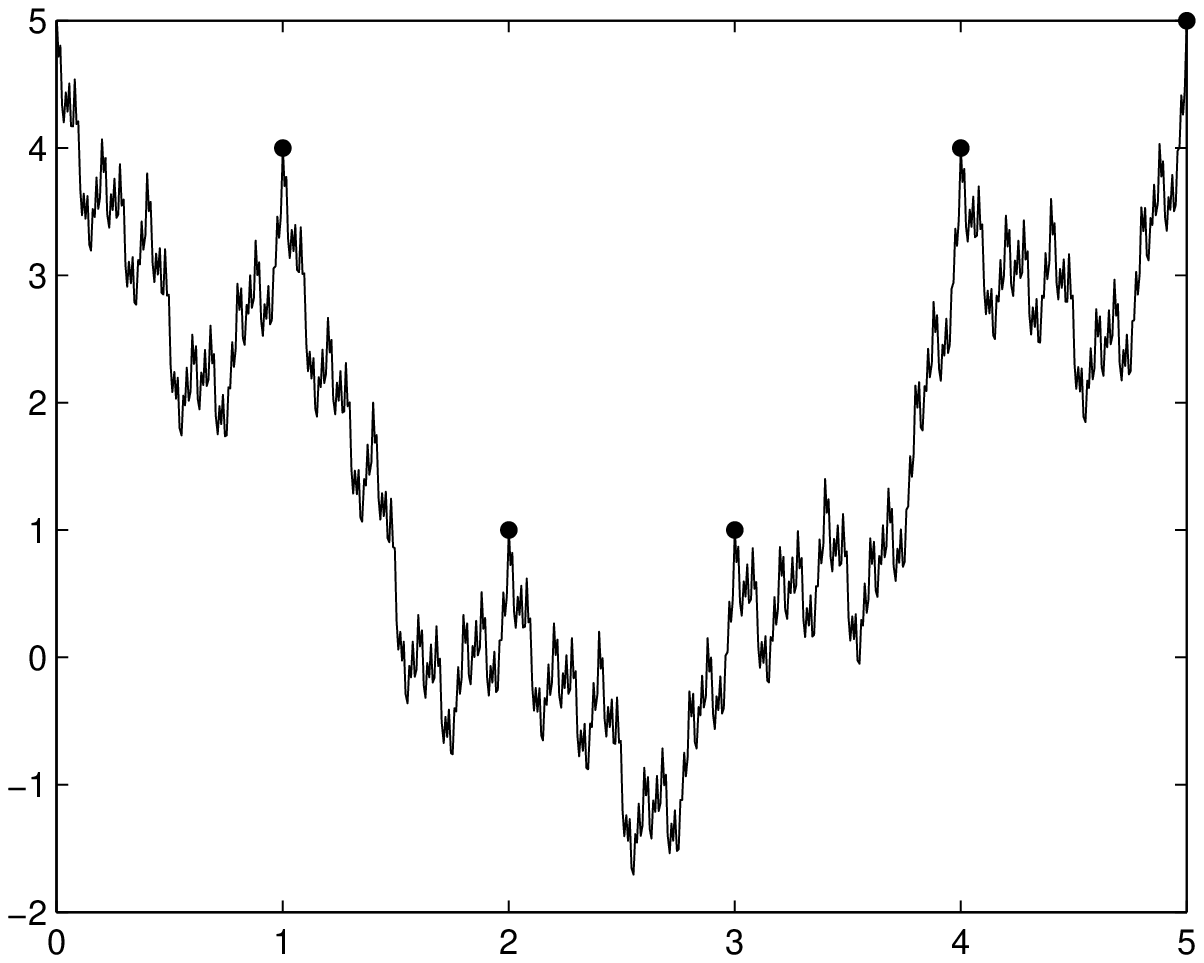}
\caption{}
\label{figure3}
\quad
\includegraphics[width=7cm]{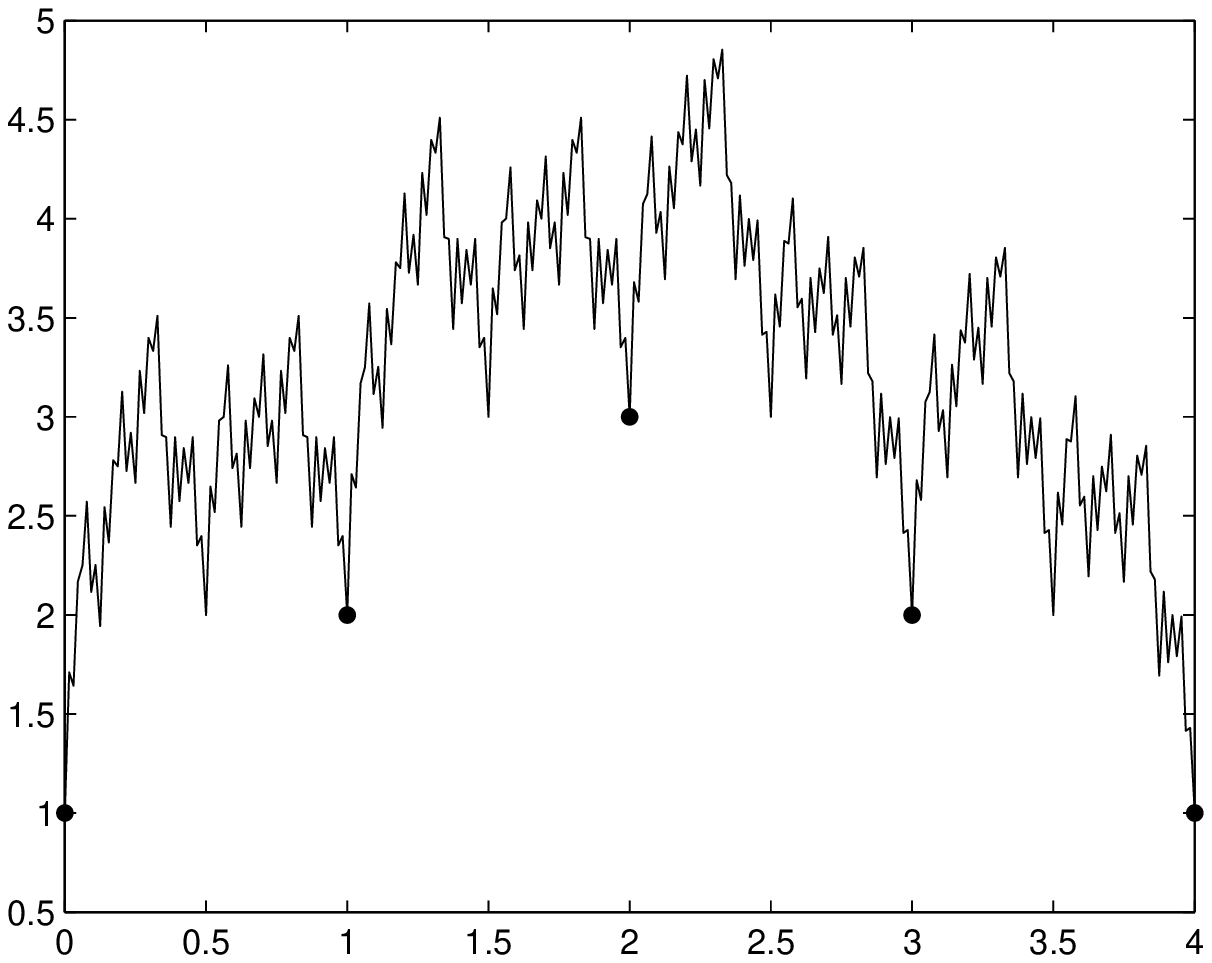}
\caption{}
\label{figure4}
\end{center}
\end{figure}

\begin{figure}[!ht]
\begin{center}
\includegraphics[width=7cm]{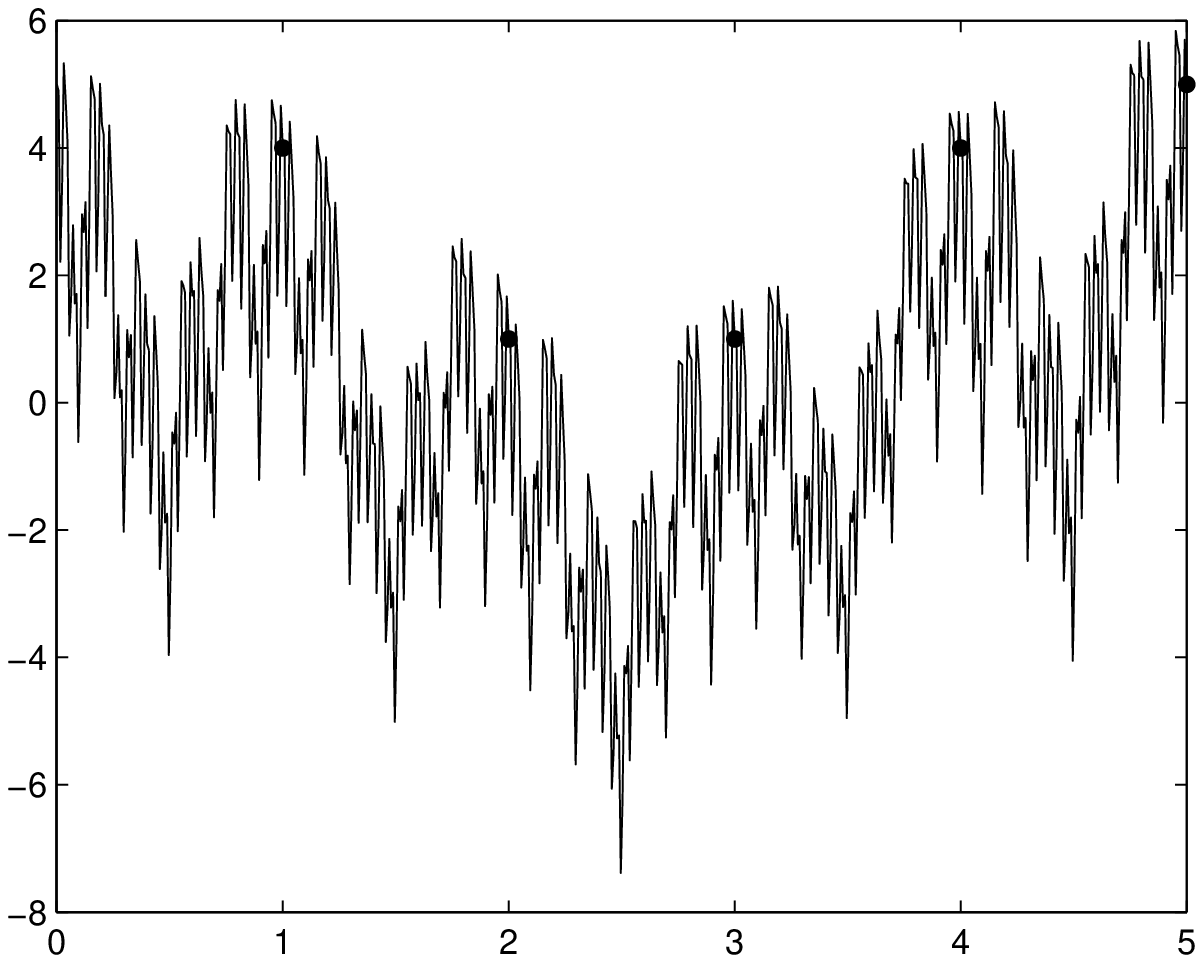}
\caption{}
\label{figure5}
\quad
\includegraphics[width=7cm]{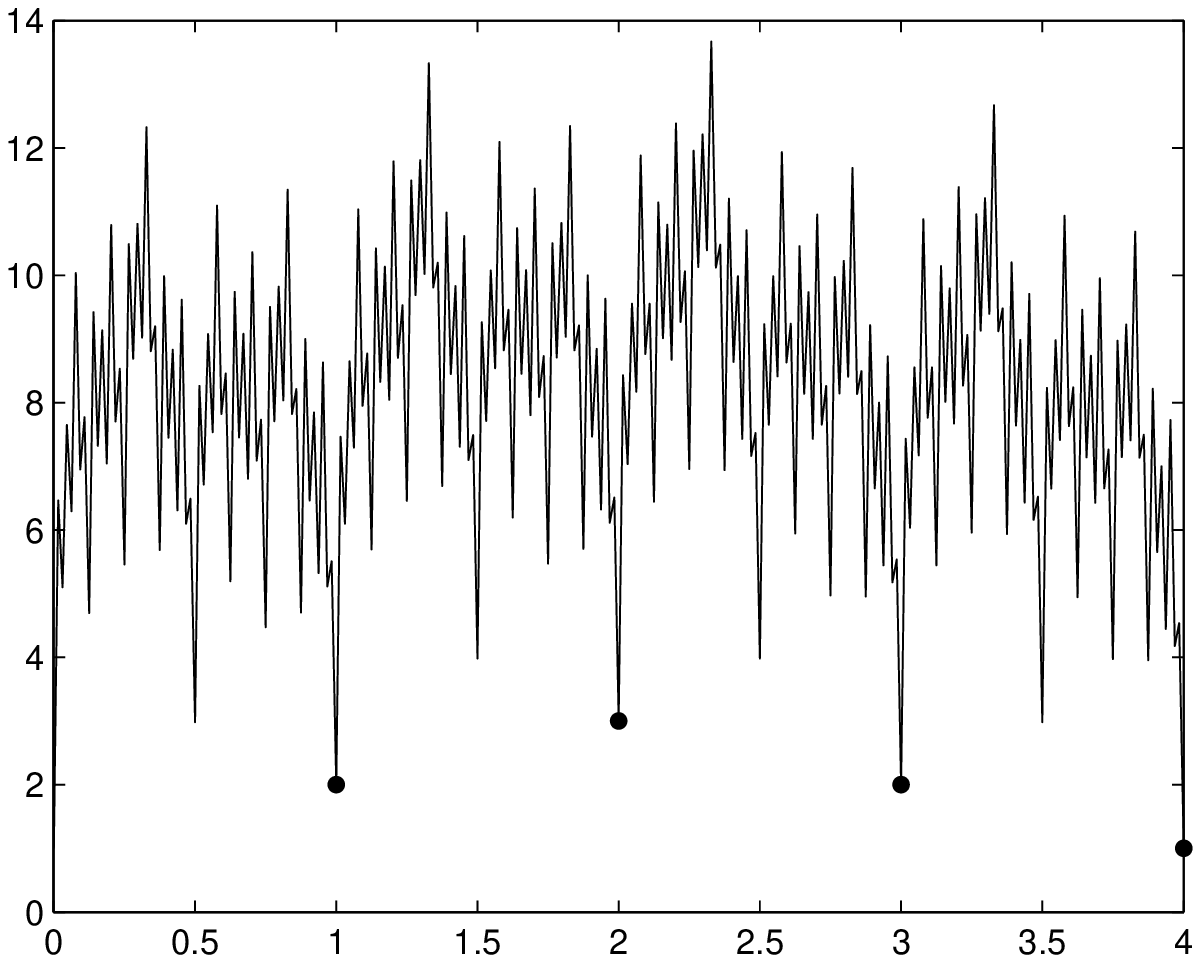}
\caption{}
\label{figure6}
\end{center}
\end{figure}

\baselineskip=12pt
\bibliographystyle{amsplain}
\bibliography{references}
\end{document}